\documentclass[11pt]{amsart}

\usepackage[margin = 1.0in]{geometry}

\usepackage{amsmath,amssymb,amsthm}
\usepackage{hyperref}

\newtheorem{theorem}{Theorem}[section]
\newtheorem{lemma}[theorem]{Lemma}
\newtheorem{proposition}[theorem]{Proposition}
\newtheorem{corollary}[theorem]{Corollary}

\theoremstyle{remark}
\newtheorem{remark}[theorem]{Remark}

\newcommand{\C}{\mathbb C}
\newcommand{\R}{\mathbb R}
\newcommand{\N}{\mathbb N}
\newcommand{\Disk}{\mathbb D}
\newcommand{\eps}{\varepsilon}
\newcommand{\diam}{\operatorname{diam}}
\newcommand{\spec}{\operatorname{spec}}
\newcommand{\ord}{\operatorname{ord}}
\newcommand{\osc}{\operatorname{osc}}
\newcommand{\rad}{\operatorname{rad}}
\newcommand{\esssup}{\operatorname*{ess sup}}

\numberwithin{equation}{section}

\title[Disk-Growth Remez and Tur\'an-Nazarov]{A Disk-Growth Remez Principle and a Modular Proof of the Measurable Tur\'an-Nazarov Inequality}

\author[Friedland]{Omer Friedland}
\address{Institut de Math\'ematiques de Jussieu, Sorbonne Universit\'e, 4 place Jussieu, 75005 Paris, France}
\email{omer.friedland@imj-prg.fr}
\date{\today}

\subjclass[2020]{Primary 41A17; Secondary 30A10, 42A05}
\keywords{Remez inequalities, Tur\'an-Nazarov inequality, exponential
polynomials, disk growth, Cartan coverings, geometric means, logarithmic
derivatives}

\begin{document}

\begin{abstract}
We give a modular proof of the measurable Tur\'an-Nazarov inequality for
exponential polynomials. The proof first establishes a Remez principle for
holomorphic functions satisfying two disk-growth assumptions. The global growth
assumption controls the number of relevant zeros, while the local growth
assumption gives an effective degree. This yields Cartan coverings, sublevel
estimates, and a geometric-mean Remez inequality.

For exponential polynomials with bounded spectral diameter, the required disk
growth follows from the classical interval Tur\'an inequality. For large
spectral diameter, we use a first-order pruning step. If
$\rho = \diam(\spec p)$ and $a\in\spec p$, then
$$
Q_a = \rho^{-1}(D-a)p
$$
has one fewer exponential term, and the quotient $Q_a/p$ satisfies an absolute
weak distribution estimate away from the zero set of $p$.

Writing
$$
Q_a = \rho^{-1}(D-a)p,
\quad
Q_b = \rho^{-1}(D-b)p
$$
for two farthest spectral points $a,b$ gives
$$
Q_a-Q_b = \frac{b-a}{\rho}p,
\quad |b-a| = \rho,
$$
and hence $|p|\le |Q_a|+|Q_b|$. The induction is carried out in
geometric-mean form on the original measurable set. This avoids losing a fixed
proportion of the set at each step and gives the classical measurable
Tur\'an-Nazarov inequality with the sharp algebraic exponent $m-1$. The final
measurable $L^\infty$ estimate is classical; the point here is the modular
proof and the geometric-mean induction. The only Tur\'an-type input is the
classical interval Tur\'an inequality.
\end{abstract}

\maketitle

%%%%%%%%%%%%%%%%%%%%%%%%%%%%%%%%%%%%%%%%%%%
\section{Introduction and main results}

Remez-type inequalities compare the size of a function on a large set with its
size on a smaller controlling set. The classical theorem of Remez gives such a
comparison for algebraic polynomials on intervals. Multidimensional and
geometric variants were developed, among others, by Brudnyi and Ganzburg,
Ganzburg, Yomdin, Brudnyi, and Brudnyi-Yomdin; see
\cite{Remez,BrudnyiGanzburg,Ganzburg,YomdinDiscrete,Brudnyi,BrudnyiYomdin}
and the references there. For exponential polynomials the corresponding
measurable result is the Tur\'an-Nazarov inequality, originating in Tur\'an's
interval lemma and Nazarov's local estimates for exponential sums; see
\cite{Turan,Turan1984,Nazarov,NazarovComplete}. The geometric refinements of
Friedland and Yomdin replace Lebesgue measure in Tur\'an-Nazarov by a metric
span-type invariant, and are another motivation for the present point of view
\cite{FriedlandYomdin}.

The purpose of this note is to give a modular proof of the measurable
Tur\'an-Nazarov theorem. The final measurable $L^\infty$ estimate is
classical; the contribution of the present note is the proof architecture: a
disk-growth Remez principle, a weak logarithmic-derivative pruning step, and a
geometric-mean induction which avoids iterative losses of measure. The proof
first yields a geometric-mean form and then the usual measurable $L^\infty$
form. It has three steps. First, when the spectrum has bounded diameter, the
classical interval Tur\'an inequality implies disk growth. A disk-growth Remez
principle then gives a geometric-mean Remez estimate. Second, when the spectral
diameter is large, one can prune one exponent by applying a first-order
operator. If $a\in\spec p$ and $\rho = \diam(\spec p)$, then
$$
Q_a = \rho^{-1}(D-a)p
$$
has order one less than $p$, and the quotient $Q_a/p$ satisfies an absolute
weak distribution estimate on the normalized interval. Third, choosing two
farthest exponents $a,b$ gives
$$
Q_a-Q_b = \frac{b-a}{\rho}p,
\quad |b-a| = \rho,
$$
and hence
$$
|p|\le |Q_a|+|Q_b|.
$$
The induction is run in geometric-mean form on the original measurable set.
This is the point that prevents a loss of a fixed proportion of the set at each
order-reduction step.

We emphasize the non-circularity of the argument. The only Tur\'an-type input
is the classical interval Tur\'an inequality. It is used first to obtain disk
growth in the bounded-diameter case, and later, through the same disk-growth
consequence, to count zeros in the proof of the weak logarithmic-derivative
estimate. No measurable Tur\'an-Nazarov inequality is used as an input.

We now introduce notation. Let
$$
I_0 = [-1/2,1/2],
\quad
\Disk(a,r) = \{z\in\C:|z-a|<r\}.
$$
A reduced exponential polynomial is a finite sum
$$
p(t) = \sum_{\lambda\in\Lambda}c_\lambda e^{\lambda t},
\quad c_\lambda\ne0,
$$
with distinct exponents. We write
$$
\spec p = \Lambda,
\quad
\ord p = |\Lambda|.
$$
For the zero polynomial we set $\spec(0) = \varnothing$ and $\ord(0) = 0$.
Throughout, $D = d/dt$. The letters $C,c,C_0,C_1,\ldots$ denote positive
absolute constants whose values may change from line to line unless explicitly
fixed.

Zeros are counted with multiplicity. Zeros in closed disks are always counted
with the standard limiting convention: one first counts zeros in slightly
smaller or larger disks and then lets the auxiliary radius tend to the desired
value.

If $E\subset\R$ is measurable with $|E|>0$, and if $g$ is a nonzero
exponential polynomial, define
$$
G_E(g) = \exp\left(\frac1{|E|}\int_E\log |g(t)| dt\right).
$$
The zero set of a nonzero exponential polynomial is finite on compact real
intervals, unless the polynomial is identically zero. After redefining
$\log |g|$ on this finite set, the logarithmic singularities are integrable.
We use the same notation whenever $\log |g|\in L^1(E)$ with this harmless
finite-set convention; this will be checked explicitly in the few places where
$g$ is not an exponential polynomial.

The first result is the abstract disk-growth Remez principle. Its proof in
Section \ref{sec:dg} is short because the elementary Blaschke,
Harnack, and Cartan details are collected in Appendix \ref{app:bc}.

\begin{theorem}[Disk-growth Remez principle] \label{thm:dg}
Let $f\not\equiv0$ be holomorphic in $\Disk(0,5)$. Let $d\in\N$,
$d\ge1$, and let $H_0\ge1$, $H_1\ge d$. Assume
\begin{align}
\sup_{\Disk(0,4)}|f|
&\le e^{H_0}\sup_{\Disk(0,1)}|f| \label{G0}\\
\sup_{\Disk(x,4)}|f|
&\le e^{H_1}r^{-d}\sup_{\Disk(x,r)}|f|,
\quad x\in I_0,
\quad 0<r\le1. \label{G1}
\end{align}
Then, for such an $f$, the zero set on $I_0$ is finite. After redefining
$\log |f|$ on this finite set, the logarithmic singularities are integrable;
thus $G_E(f)$ is well-defined for every measurable $E\subset I_0$ with
$|E|>0$. For every such set, writing $|E| = \alpha$, one has
$$
\sup_{I_0}|f| \le e^{C(H_0+H_1)}
\left(\frac C\alpha\right)^d G_E(f).
$$
Consequently,
$$
\sup_{I_0}|f| \le e^{C(H_0+H_1)} \left(\frac C\alpha\right)^d \esssup_E |f|.
$$
Moreover, for every $0<h\le1$, there exists
$\Omega_h\subset I_0$, representable as a union of at most
$\lceil C H_0\rceil$ intervals, such that $|\Omega_h|\le Ch$ and
$$
|f(x)| \ge e^{-C(H_0+H_1)}h^d\sup_{I_0}|f|, \quad x\in I_0\setminus\Omega_h.
$$
In particular, for $0<\eps\le1$,
$$
\left|
\left\{x\in I_0:
|f(x)|\le\eps\sup_{I_0}|f|
\right\}
\right| \le
C\exp\left(C\frac{H_0+H_1}{d}\right)\eps^{1/d}.
$$
\end{theorem}

The main exponential-polynomial theorem is the following geometric-mean form.

\begin{theorem}[Geometric measurable Tur\'an-Nazarov estimate] \label{thm:gtn}
Let $m\ge1$, and let
$$
p(t) = \sum_{j = 1}^m c_j e^{\lambda_jt},
\quad \lambda_j\in\C,
$$
be a nonzero exponential polynomial of order at most $m$. After reducing the
representation, put
$$
M_{\Re} = \max_{\lambda\in\spec p} |\operatorname{Re}\lambda|.
$$
Let $I\subset\R$ be a bounded interval, and let $E\subset I$ be measurable
with $|E|>0$. Then
$$
\sup_I |p| \le
\exp(|I|M_{\Re})
\left(\frac{C|I|}{|E|}\right)^{m-1} G_E(p).
$$
The constant $C$ is absolute.
\end{theorem}

Since $G_E(p)\le\esssup_E|p|$, this immediately gives the usual measurable
form.

\begin{corollary}[Measurable Tur\'an-Nazarov estimate] \label{cor:tn}
Let $m\ge1$, and let
$$
p(t) = \sum_{j = 1}^m c_j e^{\lambda_jt},
\quad \lambda_j\in\C,
$$
be an exponential polynomial of order at most $m$. Let $I\subset\R$ be a
bounded interval, and let $E\subset I$ be measurable with $|E|>0$. If
$p\equiv0$, the conclusion is trivial. Otherwise, after reducing the
representation, put
$$
M_{\Re} = \max_{\lambda\in\spec p} |\operatorname{Re}\lambda|.
$$
Then
$$
\sup_I |p| \le
\exp(|I|M_{\Re})
\left(\frac{C|I|}{|E|}\right)^{m-1}
\esssup_E |p|.
$$
The constant $C$ is absolute.
\end{corollary}

No sharpness is claimed for the absolute constant in the algebraic base. The
exponent $m-1$, however, is sharp in the usual sense, already on
intervals; see Remark \ref{rem:sharp}.

%%%%%%%%%%%%%%%%%%%%%%%%%%%%%%%%%%%%%%%%%%%
\section{A disk-growth Remez lemma} \label{sec:dg}

The proof of Theorem \ref{thm:dg} is divided into two compact
propositions. The first turns the two disk-growth assumptions into an effective
zero-product lower bound. The global disk-growth assumption controls the number
of zeros; the local disk-growth assumption controls how many of those zeros can
matter at a fixed real point.

The point of the next proposition is that the two growth assumptions play
different roles. The global bound \eqref{G0} gives a bound for the total
number of zeros in $\Disk(0,2)$. The local bound \eqref{G1}, with the
power $r^{-d}$, prevents more than $d$ zeros near a real point from
contributing without compensation. Thus the final lower bound involves only
the first $d$ nearest zeros, although the total number of zeros is controlled
by $H_0$.

\begin{proposition}[Disk growth implies an effective zero product] \label{prop:zp}
Let $f\not\equiv0$ be holomorphic in $\Disk(0,5)$, and assume
\eqref{G0} and \eqref{G1}. Normalize
$$
 \sup_{I_0}|f| = 1.
$$
Let $\mathcal Z$ be the multiset of zeros of $f$ in $\Disk(0,2)$, counted
with multiplicity, and put $N = |\mathcal Z|$. Then
$$
 N\le C H_0.
$$
Moreover, for every $x\in I_0$, if
$$
 |x-z_1(x)|\le\cdots\le |x-z_N(x)|
$$
is the ordering of $\mathcal Z$ by distance from $x$, then
$$
 |f(x)|
 \ge
 e^{-C(H_0+H_1)}
 \prod_{j = 1}^{\min\{d,N\}}
 \min\{1,|x-z_j(x)|\}.
$$
If $N = 0$, the product is interpreted as $1$.
\end{proposition}

The proof is given in Appendix \ref{app:zp}. It is a standard
Blaschke-product argument, together with one Harnack estimate for a zero-free
quotient and a local tail-product estimate.

The second proposition is the real-variable Cartan/Remez step. It says that a
lower bound by the first $d$ nearest zeros already contains the full
sublevel-set and Remez information.

\begin{proposition}[Product Cartan-Remez lemma] \label{prop:cr}
Let $F$ be a bounded measurable function on $I_0$ with
$\sup_{I_0}|F| = 1$. Let $w_1,\ldots,w_N\in\C$, with repetitions allowed,
and let $d\ge1$ be an integer. For $x\in I_0$, let
$$
 d_1(x)\le\cdots\le d_N(x)
$$
be the numbers $|x-w_1|,\ldots,|x-w_N|$ arranged in nondecreasing order. If
$N = 0$, all products below are empty. Assume that, for some $S\ge0$,
\begin{equation} \label{eq:prod}
 |F(x)|
 \ge
 e^{-S}
 \prod_{j = 1}^{\min\{d,N\}}\min\{1,d_j(x)\},
\quad x\in I_0.
\end{equation}
The lower bound implies that $F$ can vanish only at the real points among
$w_1,\ldots,w_N$. After redefining $\log|F|$ on this finite set, the
same lower bound shows that $\log|F|\in L^1(E)$ for every measurable
$E\subset I_0$. We define $G_E(F)$ with this harmless convention.

Then, for every $0<h\le1$, there exists $\Omega_h\subset I_0$,
representable as a union of at most $N$ intervals, such that
$$
 |\Omega_h|\le Ch
$$
and
$$
 |F(x)|\ge e^{-C(S+N)}h^d,
\quad x\in I_0\setminus\Omega_h.
$$
Consequently, for $0<\eps\le1$,
$$
 \left|\{x\in I_0: |F(x)|\le\eps\}\right|
 \le
 C\exp\left(C\frac{S+N}{d}\right)\eps^{1/d}.
$$
Finally, if $E\subset I_0$ is measurable with $|E| = \alpha>0$, then
$$
 1
 \le
 e^{C(S+N)}
 \left(\frac C\alpha\right)^d
 G_E(F).
$$
\end{proposition}

\begin{proof}
Apply the Cartan covering lemma, Lemma \ref{lem:cartan}, to the
points $w_1,\ldots,w_N$. For $0<h\le1$, it gives at most $N$ disks
whose total radii are at most $Ch$ and outside which
$$
 \prod_{j = 1}^{\min\{d,N\}}\min\{1,d_j(x)\}
 \ge e^{-CN}h^d.
$$
Intersecting the disks with $I_0$ gives a union $\Omega_h$ of at most $N$
intervals with $|\Omega_h|\le Ch$. Combining the last display with
\eqref{eq:prod} gives
$$
 |F(x)|\ge e^{-C(S+N)}h^d,
\quad x\in I_0\setminus\Omega_h.
$$

Put $A = C(S+N)$. For a sufficiently large absolute constant $B$, choose
$$
 h = B\exp\left(B\frac{S+N}{d}\right)\eps^{1/d}.
$$
If $h\ge1$, then the claimed sublevel bound is trivial after enlarging the
absolute constant. If $h<1$, then outside $\Omega_h$ we have
$|F|>\eps$, provided $B$ is chosen larger than the implicit constant in
$A$. Hence the sublevel set is contained in $\Omega_h$, which proves the
sublevel estimate.

For the geometric-mean estimate, set
$$
 Y(x) = \log\frac1{|F(x)|}
$$
where $F(x)\ne0$. The lower product bound implies that the zero set of
$F$ is contained in the finite set of real points among the $w_j$'s. After
changing $Y$ on this finite set, the same lower bound also implies that $Y$ is
integrable. By the distribution formula and the sublevel estimate,
$$
\begin{aligned}
 \frac1\alpha\int_E Y(x) dx
 &\le
 \int_0^\infty
 \min\left\{1,
 \frac C\alpha
 \exp\left(C\frac{S+N}{d}\right)e^{-s/d}
 \right\} ds \\
 &\le C(S+N)+d\log\frac C\alpha,
\end{aligned}
$$
after changing $C$. Since
$$
 G_E(F) = \exp\left(-\frac1\alpha\int_EY\right),
$$
we obtain
$$
 1\le e^{C(S+N)}\left(\frac C\alpha\right)^dG_E(F).
$$
This completes the proof.
\end{proof}

\begin{proof}[Proof of Theorem \ref{thm:dg}]
Since $f\not\equiv0$, the identity theorem gives $\sup_{I_0}|f|>0$. By
homogeneity, normalize $\sup_{I_0}|f| = 1$. Proposition
\ref{prop:zp} gives a multiset $\mathcal Z$ of zeros in
$\Disk(0,2)$, with $N\le C H_0$, and a product lower bound of the form
\eqref{eq:prod} with
$$
 S = C(H_0+H_1).
$$
Applying Proposition \ref{prop:cr} gives the Cartan covering,
the sublevel estimate, and
$$
 1
 \le
 e^{C(H_0+H_1)}
 \left(\frac C\alpha\right)^d
 G_E(f).
$$
Undoing the normalization gives the geometric-mean Remez estimate. The
measurable estimate follows from $G_E(f)\le\esssup_E|f|$.
\end{proof}

%%%%%%%%%%%%%%%%%%%%%%%%%%%%%%%%%%%%%%%%%%%
\section{Bounded spectral diameter}

The bounded-spectrum part uses only the following classical interval form of
Tur\'an's first main theorem for exponential sums. We record it as an input.
This is the only Tur\'an-type input in the paper; in particular, we do not use
the measurable Tur\'an-Nazarov theorem.

\begin{theorem}[Classical interval Tur\'an inequality] \label{thm:turan}
Let $p\not\equiv0$ be an exponential polynomial of order at most $m$. After
reducing the representation, write
$$
 p(t) = \sum_{\lambda\in\Lambda} c_\lambda e^{\lambda t},
\quad |\Lambda|\le m,
$$
and put
$$
 M = \max_{\lambda\in\Lambda} |\operatorname{Re}\lambda|.
$$
If $J\subset K\subset\R$ are bounded intervals and $|J|>0$, then
$$
 \sup_K |p|
 \le
 \exp(|K|M)
 \left(\frac{C|K|}{|J|}\right)^{m-1}
 \sup_J |p|.
$$
The constant $C$ is absolute.
\end{theorem}

This is the usual interval form of Tur\'an's first main theorem for exponential
sums; see \cite{Turan,Turan1984,MitrinovicPecaricFink}. In this note it is the
only Tur\'an-type input. It enters the argument through Lemma \ref{lem:edg};
therefore the disk zero-counting lemma in Appendix \ref{app:wld}, which uses
Lemma \ref{lem:edg}, also ultimately depends on this same classical input.

\begin{lemma}[Interval Tur\'an implies disk growth] \label{lem:edg}
Let
$$
 q(z) = \sum_{k = 1}^m a_k e^{\mu_kz}
$$
be a nonzero exponential polynomial of order at most $m$, and put
$$
 \sigma = \max_k |\mu_k|.
$$
Then, for every $z_0\in\C$ and every $0<r<R$,
$$
 \sup_{\Disk(z_0,R)}|q|
 \le
 \exp(2R\sigma)
 \left(\frac{CR}{r}\right)^{m-1}
 \sup_{\Disk(z_0,r)}|q|.
$$
\end{lemma}

\begin{proof}
Fix $z\in\Disk(z_0,R)$. If $z = z_0$, choose any $|\theta| = 1$ and put
$s_0 = 0$; otherwise write $z = z_0+s_0\theta$, with $0<s_0<R$ and
$|\theta| = 1$. For real $s$, define
$$
 Q(s) = q(z_0+s\theta).
$$
Then $Q$ is an exponential polynomial in the real variable $s$, with
exponents $\theta\mu_k$, and
$$
 |\operatorname{Re}(\theta\mu_k)|\le\sigma.
$$
Moreover $Q\not\equiv0$; otherwise $q$ would vanish on a line segment, and
analytic continuation would give $q\equiv0$.
Apply Theorem \ref{thm:turan} to $[-r,r]\subset[-R,R]$. Since
$z_0+s\theta\in\Disk(z_0,r)$ for $|s|\le r$, we get
$$
 |q(z)| = |Q(s_0)|
 \le
 \exp(2R\sigma)
 \left(\frac{CR}{r}\right)^{m-1}
 \sup_{\Disk(z_0,r)}|q|.
$$
Taking the supremum over $z\in\Disk(z_0,R)$ proves the lemma.
\end{proof}

We shall also use the following elementary bookkeeping lemma for multiplication
by a single exponential factor.

\begin{lemma}[Single exponential factor] \label{lem:exp}
Let $E\subset I_0$ be measurable with $|E|>0$, let $p$ be a nonzero
measurable function on $I_0$ with $\log |p|\in L^1(E)$, and let
$\lambda_0\in\C$. Put
$$
 F(t) = e^{-\lambda_0t}p(t).
$$
Then
$$
 \sup_{I_0}|p|
 \le
 e^{|\operatorname{Re}\lambda_0|}
 \frac{\sup_{I_0}|F|}{G_E(F)}G_E(p).
$$
In particular,
$$
 \sup_{I_0}|e^{\lambda t}|
 \le
 e^{|\operatorname{Re}\lambda|}G_E(e^{\lambda t}).
$$
\end{lemma}

\begin{proof}
Let $\alpha = |E|$. Since both $t\in I_0$ and the average
$\alpha^{-1}\int_Et dt$ lie in $[-1/2,1/2]$,
$$
 \sup_{t\in I_0}\operatorname{Re}\lambda_0 t
 -\frac1\alpha\int_E\operatorname{Re}\lambda_0 t dt
 \le |\operatorname{Re}\lambda_0|.
$$
Also
$$
 G_E(F)
 =
 \exp\left(-\frac1\alpha\int_E\operatorname{Re}\lambda_0 t dt\right)
 G_E(p).
$$
Combining this identity with
$$
 \sup_{I_0}|p|
 \le
 \exp\left(\sup_{t\in I_0}\operatorname{Re}\lambda_0 t\right)
 \sup_{I_0}|F|
$$
gives the first assertion. The second follows by taking $p = e^{\lambda t}$
and $F\equiv1$.
\end{proof}

\begin{corollary}[Bounded-spectrum geometric Remez] \label{cor:bs}
Let $m\ge1$, and let
$$
 p(t) = \sum_{j = 1}^m c_j e^{\lambda_jt}
$$
be a nonzero exponential polynomial of order at most $m$. Put
$$
 \sigma = \max_j |\lambda_j|.
$$
If $E\subset I_0$ is measurable and $|E| = \alpha>0$, then
$$
 \sup_{I_0}|p|
 \le
 e^{C\sigma}
 \left(\frac C\alpha\right)^{m-1}
 G_E(p).
$$
Consequently,
$$
 \sup_{I_0}|p|
 \le
 e^{C\sigma}
 \left(\frac C\alpha\right)^{m-1}
 \esssup_E |p|.
$$
\end{corollary}

\begin{proof}
Reduce the representation by combining equal exponents and deleting zero
coefficients. Let $n = \ord p\le m$. The spectral radius of the reduced
representation is at most the original $\sigma$. If $n = 1$, the result
follows from Lemma \ref{lem:exp}.

Assume $n\ge2$ and set $d = n-1$. By Lemma \ref{lem:edg},
\eqref{G0} holds for $p$ with
$$
 H_0 = C(\sigma+n),
$$
and \eqref{G1} holds with
$$
 H_1 = C(\sigma+n),
\quad d = n-1.
$$
Indeed, for \eqref{G0} use $R = 4$, $r = 1$, and for \eqref{G1} use the
same outer radius $R = 4$ and arbitrary $0<r\le1$. Increasing $C$ if
necessary, $H_0\ge1$ and $H_1\ge d$.

Theorem \ref{thm:dg} gives, with an absolute constant $C_1$,
$$
 \sup_{I_0}|p|
 \le
 e^{C_1(\sigma+n)}
 \left(\frac {C_1}\alpha\right)^{n-1}
 G_E(p).
$$
Since $n\ge2$,
$$
 e^{C_1n}
 \le
 \left(e^{2C_1}\right)^{n-1}.
$$
Therefore, for a new absolute constant $C_2$,
$$
 e^{C_1n}\left(\frac {C_1}\alpha\right)^{n-1}
 \le
 \left(\frac {C_2}\alpha\right)^{n-1}.
$$
This is the only absorption of an order-dependent exponential factor in this
corollary. Also $n\le m$ and $0<\alpha\le1$. Hence, after changing the
absolute constant once more,
$$
 \sup_{I_0}|p|
 \le
 e^{C\sigma}
 \left(\frac C\alpha\right)^{m-1}
 G_E(p).
$$
The measurable form follows from $G_E(p)\le\esssup_E|p|$.
\end{proof}

%%%%%%%%%%%%%%%%%%%%%%%%%%%%%%%%%%%%%%%%%%%
\section{Large-diameter pruning}

The large-diameter step rests on a weak logarithmic-derivative estimate. Its
proof is placed in Appendix \ref{app:wld}. The only real-variable input is the
weak $(1,1)$ inequality for the Hilbert transform, applied to the Cauchy sum
over the zeros; the remaining ingredients are disk zero counting and the
zero-free quotient estimate from Appendix \ref{app:bc}.

\begin{lemma}[Weak logarithmic derivative] \label{lem:wld}
Let
$$
 q(t) = \sum_{k = 1}^m a_k e^{\mu_k t}
$$
be a nonzero exponential polynomial of order at most $m$, and put
$$
 \sigma = \max_k |\mu_k|.
$$
Then, for every $u>0$,
$$
 \left|
 \left\{t\in I_0:
 q(t)\ne0,
 |q'(t)/q(t)|>u
 \right\}
 \right|
 \le
 \frac{C(m+\sigma)}{u}.
$$
\end{lemma}

\begin{lemma}[Large-diameter pruning] \label{lem:prune}
Let
$$
 q(t) = \sum_{\lambda\in\Gamma}b_\lambda e^{\lambda t}
$$
be reduced and of exact order $n\ge2$. Put
$$
 \rho = \diam\Gamma.
$$
Assume
$$
 \rho>A_0n
$$
for a fixed absolute constant $A_0\ge1$. Let $a\in\Gamma$, and define
$$
 Q_a = \rho^{-1}(D-a)q.
$$
Then, for every $u>0$,
$$
 \left|
 \left\{t\in I_0:
 q(t)\ne0,
 |Q_a(t)|>u|q(t)|
 \right\}
 \right|
 \le
 \frac Cu,
$$
where $C$ is absolute once $A_0$ has been fixed.
\end{lemma}

\begin{proof}
Put $g(t) = e^{-at}q(t)$. Then $g$ has order $n$ and exponents
$\lambda-a$, $\lambda\in\Gamma$, whose moduli are at most $\rho$. Away
from the zero set of $q$,
$$
 \frac{g'(t)}{g(t)} = \frac{(D-a)q(t)}{q(t)}.
$$
Therefore
$$
 |Q_a(t)|>u|q(t)|
\quad\Longleftrightarrow\quad
 \left|\frac{g'(t)}{g(t)}\right|>u\rho.
$$
Lemma \ref{lem:wld}, applied to $g$ with threshold
$u\rho$, gives
$$
 \left|
 \left\{t\in I_0:
 q(t)\ne0,
 |Q_a(t)|>u|q(t)|
 \right\}
 \right|
 \le
 \frac{C(n+\rho)}{u\rho}
 \le \frac Cu,
$$
because $\rho>A_0n$.
\end{proof}

\begin{lemma}[Geometric quotient control] \label{lem:gq}
Under the hypotheses of Lemma \ref{lem:prune}, let
$E\subset I_0$ be measurable with $|E| = \alpha>0$. Then
$$
 G_E(Q_a)
 \le
 \frac C\alpha G_E(q),
$$
where $C$ is absolute once $A_0$ has been fixed.
\end{lemma}

\begin{proof}
Under the hypotheses of Lemma \ref{lem:prune}, $Q_a\not\equiv0$, since
$D-a$ kills exactly the $a$-term and leaves all other spectral terms with
nonzero coefficients. Thus $q$ and $Q_a$ are nonzero exponential polynomials,
and their zeros on $I_0$ are finite. Hence $\log|q|$ and $\log|Q_a|$ are
locally integrable on $I_0$. Define
$$
 R_a(t) = \frac{|Q_a(t)|}{|q(t)|}
$$
where $q(t)\ne0$, and define it arbitrarily on the finite zero set of
$q$. Changing $R_a$ on this finite set does not affect distribution functions
or integrals. By Lemma \ref{lem:prune},
$$
 |\{t\in I_0:R_a(t)>u\}|\le \frac Cu,
\quad u>0.
$$
Using the distribution formula for $\log_+R_a$,
$$
\begin{aligned}
\frac1\alpha\int_E\log_+R_a(t) dt \le \int_0^\infty
\min\left\{1,\frac{C}{\alpha e^s}\right\} ds \le \log\frac C\alpha.
\end{aligned}
$$
Since
$$
\log |Q_a(t)| \le \log |q(t)|+\log_+R_a(t)
$$
outside a finite set, integration over $E$ gives
$$
\frac1\alpha\int_E\log |Q_a(t)| dt \le \frac1\alpha\int_E\log |q(t)| dt +\log\frac C\alpha.
$$
Exponentiating proves the claim.
\end{proof}

%%%%%%%%%%%%%%%%%%%%%%%%%%%%%%%%%%%%%%%%%%%
\section{Proof of the Tur\'an-Nazarov theorem}

We first prove the normalized geometric form on $I_0$. The general interval
case follows by scaling.

\begin{theorem}[Normalized geometric Tur\'an-Nazarov estimate] \label{thm:ntn}
Let $m\ge1$, and let
$$
 p(t) = \sum_{j = 1}^m c_j e^{\lambda_jt}
$$
be a nonzero exponential polynomial of order at most $m$. After reducing the
representation, put
$$
 M_{\Re} = \max_{\lambda\in\spec p} |\operatorname{Re}\lambda|.
$$
Let $E\subset I_0$ be measurable with $|E| = \alpha>0$. Then
$$
 \sup_{I_0}|p|
 \le
 \exp(M_{\Re})
 \left(\frac C\alpha\right)^{m-1}
 G_E(p).
$$
The constant $C$ is absolute.
\end{theorem}

\begin{proof}
Delete zero coefficients and combine equal exponents. It is enough to prove
the estimate with $m = \ord p$, because $0<\alpha\le1$. Thus assume that
$$
 p(t) = \sum_{\lambda\in\Lambda}c_\lambda e^{\lambda t},
\quad |\Lambda| = m,
$$
is reduced and of exact order $m$, that is, it has exactly $m$ distinct
exponential terms. Then
$$
 M_{\Re} = \max_{\lambda\in\Lambda}|\operatorname{Re}\lambda|.
$$

Fix an absolute constant $A_0\ge1$. Let $C_{\rm bs}$ be chosen large
enough to serve as the constant in both places in Corollary \ref{cor:bs},
and let $C_{\rm gq}$ be a constant for which Lemma
\ref{lem:gq} holds with this fixed $A_0$. Choose
$C_{\rm TN}\ge1$ so large that
$$
 C_{\rm TN}\ge C_{\rm bs}e^{2C_{\rm bs}A_0},
\quad
 C_{\rm TN}\ge2C_{\rm gq}.
$$
We prove the theorem by induction on $m$, with $C = C_{\rm TN}$.

If $m = 1$, the result follows from Lemma
\ref{lem:exp}. Assume now that $m\ge2$ and that the
claim has been proved for all exact orders smaller than $m$.

First suppose that
$$
 \diam\Lambda\le A_0m.
$$
Choose $\lambda_0\in\Lambda$, and put
$$
 F(t) = e^{-\lambda_0t}p(t).
$$
The spectrum of $F$ is contained in $\Disk(0,A_0m)$. Corollary
\ref{cor:bs} gives
$$
 \sup_{I_0}|F|
 \le
 e^{C_{\rm bs}A_0m}
 \left(\frac{C_{\rm bs}}\alpha\right)^{m-1}
 G_E(F).
$$
By Lemma \ref{lem:exp},
$$
 \sup_{I_0}|p|
 \le
 e^{|\operatorname{Re}\lambda_0|}
 e^{C_{\rm bs}A_0m}
 \left(\frac{C_{\rm bs}}\alpha\right)^{m-1}
 G_E(p).
$$
Since $m\ge2$,
$$
 e^{C_{\rm bs}A_0m}
 \le
 e^{2C_{\rm bs}A_0(m-1)}.
$$
Therefore
$$
 e^{C_{\rm bs}A_0m}
 \left(\frac{C_{\rm bs}}\alpha\right)^{m-1}
 \le
 \left(\frac{C_{\rm bs}e^{2C_{\rm bs}A_0}}\alpha\right)^{m-1}.
$$
Since $|\operatorname{Re}\lambda_0|\le M_{\Re}$ and
$C_{\rm TN}\ge C_{\rm bs}e^{2C_{\rm bs}A_0}$, we get
$$
 \sup_{I_0}|p|
 \le
 e^{M_{\Re}}
 \left(\frac{C_{\rm TN}}\alpha\right)^{m-1}
 G_E(p).
$$
This proves the induction step in the moderate-diameter case.

It remains to consider the case
$$
 \rho := \diam\Lambda>A_0m.
$$
Choose $a,b\in\Lambda$ such that $|a-b| = \rho$, and define
$$
 Q_a = \rho^{-1}(D-a)p,
\quad
 Q_b = \rho^{-1}(D-b)p.
$$
Then
$$
 Q_a-Q_b = \rho^{-1}\bigl((D-a)-(D-b)\bigr)p
 = \frac{b-a}{\rho}p.
$$
Since $|b-a| = \rho$,
$$
 |p(t)|\le |Q_a(t)|+|Q_b(t)|,
\quad t\in\R.
$$
The functions $Q_a$ and $Q_b$ are nonzero exponential polynomials of exact
order $m-1$. Indeed, $D-a$ kills exactly the term with exponent $a$ and
multiplies every other term by $\lambda-a\ne0$, and no exponents merge.
Their real-part parameters are at most $M_{\Re}$.

By Lemma \ref{lem:gq}, applied on the original set $E$,
$$
G_E(Q_a)\le\frac{C_{\rm gq}}\alpha G_E(p),
\quad
G_E(Q_b)\le\frac{C_{\rm gq}}\alpha G_E(p).
$$
Applying the induction hypothesis to $Q_a$ and then using the last estimate
gives
$$
\begin{aligned}
\sup_{I_0}|Q_a|
&\le
e^{M_{\Re}}
\left(\frac{C_{\rm TN}}\alpha\right)^{m-2}
G_E(Q_a) \le
e^{M_{\Re}}
\left(\frac{C_{\rm TN}}\alpha\right)^{m-2}
\frac{C_{\rm gq}}\alpha
G_E(p) \\
&=
e^{M_{\Re}}
C_{\rm gq}
\frac{C_{\rm TN}^{m-2}}{\alpha^{m-1}}
G_E(p).
\end{aligned}
$$
The same bound holds for $Q_b$. Thus the lower-order induction contributes
$\alpha^{-(m-2)}$, and the geometric quotient control contributes exactly one
additional factor $\alpha^{-1}$; no subset of $E$ has been discarded. Hence
$$
\sup_{I_0}|p| \le
e^{M_{\Re}}
2C_{\rm gq}
\frac{C_{\rm TN}^{m-2}}{\alpha^{m-1}}
G_E(p).
$$
Since $C_{\rm TN}\ge2C_{\rm gq}$, this is
$$
\sup_{I_0}|p| \le
e^{M_{\Re}}
\left(\frac{C_{\rm TN}}\alpha\right)^{m-1}
G_E(p).
$$
The induction is complete.
\end{proof}

\begin{proof}[Proof of Theorem \ref{thm:gtn}]
Let $I\subset\R$ be a bounded interval. Since $|E|>0$, the length
$L = |I|$ is positive. Since $p$ is continuous, the supremum over $I$ equals
the supremum over its closure. Thus we may use the center of the closed
interval with the same endpoints. Let $t_0$ be this center, write
$$
t = t_0+Ls,
\quad s\in I_0,
$$
and set
$$
P(s) = p(t_0+Ls).
$$
Then the reduced spectrum of $P$ is $L\spec p$, counted after the harmless
change of coefficients. Hence the transformed real-part parameter is
$LM_{\Re}$. The transformed set is
$$
E_I = \{s\in I_0:t_0+Ls\in E\},
\quad |E_I| = \frac{|E|}{L},
$$
and $G_{E_I}(P) = G_E(p)$. Applying Theorem
\ref{thm:ntn} to $P$ gives
$$
\sup_I|p| \le
\exp(|I|M_{\Re})
\left(\frac{C|I|}{|E|}\right)^{m-1}
G_E(p).
$$
\end{proof}

\begin{proof}[Proof of Corollary \ref{cor:tn}]
The case $p\equiv0$ is trivial. Otherwise Theorem
\ref{thm:gtn} applies, and
$$
G_E(p)\le\esssup_E |p|.
$$
\end{proof}

\begin{remark}[Why the geometric-mean induction closes]
A direct $L^\infty$ induction based on removing an exceptional subset of
$E$ loses a fixed fraction of $E$ at every order-reduction step. Such a
loss is then raised to the lower-order Remez exponent and produces an unwanted
exponential factor in the order. The geometric quotient control in Lemma
\ref{lem:gq} keeps the original set $E$ and
contributes exactly one additional factor $\alpha^{-1}$, which is precisely
the factor needed when passing from order $m-1$ to order $m$.
\end{remark}

\begin{remark}[Sharpness of the final exponent] \label{rem:sharp}
The exponent $m-1$ in Corollary \ref{cor:tn} is sharp in the usual
sense, already on intervals. Indeed,
$$
p_\delta(t) = \left(\frac{e^{\delta t}-1}{\delta}\right)^{m-1}
$$
is an exponential polynomial of order at most $m$, and
$$
p_\delta\to t^{m-1}
$$
uniformly on compact intervals as $\delta\downarrow0$. Taking
$$
I = [-1,1],
\quad
E = [-\eps,\eps],
$$
and then letting $\delta\downarrow0$, any estimate with algebraic exponent
$\beta<m-1$ would imply, up to harmless absolute constants,
$$
1\le C\eps^{-\beta}\eps^{m-1}
$$
for arbitrarily small $\eps>0$, which is impossible.
\end{remark}

\begin{remark}[Higher-dimensional directions]
Multivariable Tur\'an-Nazarov inequalities on convex bodies are known in
various forms; see, for instance, \cite{BrudnyiYomdin,FontesMerz}. One can
also obtain convex-body estimates by slicing and applying the one-dimensional
inequality on lines. A different question is whether the disk-growth/Cartan
principle and the quotient-pruning mechanism admit genuinely multidimensional
analogues not based on slicing. We leave this question for future work.
\end{remark}

%%%%%%%%%%%%%%%%%%%%%%%%%%%%%%%%%%%%%%%%%%%
\appendix

%%%%%%%%%%%%%%%%%%%%%%%%%%%%%%%%%%%%%%%%%%%
\section{Blaschke, Harnack, and Cartan details} \label{app:bc}

This appendix contains the elementary analytic details used in Section
\ref{sec:dg}. As in the introduction, zeros are counted with multiplicity, and
boundary issues are handled by the standard limiting convention: prove the
assertion with radii slightly smaller or larger than the boundary radius and
then let the auxiliary radius tend to the desired value.

\subsection{Partial Blaschke products}

For $a\in\C$, $R>0$, and $\zeta\in\Disk(a,R)$, define
$$
 b_{\zeta}^{a,R}(z)
 =
 R\frac{z-\zeta}{R^2-\overline{\zeta-a}(z-a)}.
$$
This is the pullback of the standard unit-disk Blaschke factor under
$z\mapsto (z-a)/R$. In particular,
$$
 |b_{\zeta}^{a,R}(z)| = 1
\quad\text{for } |z-a| = R,
\quad
 |b_{\zeta}^{a,R}(z)|\le1
\quad\text{for } z\in\Disk(a,R).
$$
For a finite multiset $Z\subset\Disk(a,R)$, write
$$
 B_Z^{a,R}(z) = \prod_{\zeta\in Z}b_{\zeta}^{a,R}(z),
$$
with the empty product equal to $1$.

\begin{lemma}[Partial Blaschke products] \label{lem:pb}
Let $F$ be holomorphic near $\overline{\Disk(a,R)}$. Let $Z$ be a finite
multiset of zeros of $F$ in $\Disk(a,R)$, with multiplicities not exceeding
their multiplicities as zeros of $F$. Then
$F/B_Z^{a,R}$ extends holomorphically to $\Disk(a,R)$. Moreover, for every
$0<s<R$,
$$
 \sup_{\Disk(a,s)}|F|
 \le
 \left(\sup_{\Disk(a,s)}|B_Z^{a,R}|\right)
 \sup_{\Disk(a,R)}|F|.
$$
\end{lemma}

\begin{proof}
The quotient has removable singularities at the points of $Z$. On
$\partial\Disk(a,R)$, the Blaschke product has modulus one. Hence the
maximum principle gives
$$
 \sup_{\Disk(a,R)}\left|\frac{F}{B_Z^{a,R}}\right|
 \le
 \sup_{\Disk(a,R)}|F|.
$$
Multiplying by $B_Z^{a,R}$ on $\Disk(a,s)$ gives the claim.
\end{proof}

\begin{lemma}[One Blaschke factor on a smaller disk] \label{lem:bf}
Let $0<s<R$, let $\zeta\in\Disk(a,R)$, and suppose
$|\zeta-a|\le\rho<R$. Then, for every $z\in\Disk(a,s)$,
$$
 \left|R\frac{z-\zeta}{R^2-\overline{\zeta-a}(z-a)}\right|
 \le
 \frac{R(s+\rho)}{R^2-s\rho}.
$$
\end{lemma}

\begin{proof}
For $z\in\Disk(a,s)$,
$$
 |z-\zeta|\le s+\rho,
\quad
 |R^2-\overline{\zeta-a}(z-a)|\ge R^2-s\rho.
$$
The displayed estimate follows.
\end{proof}

\subsection{A zero-free quotient estimate}

The following lemma packages the Harnack and Borel-Carath\'eodory estimates
used twice in the paper.

\begin{lemma}[Zero-free quotient estimates] \label{lem:zf}
Let $h$ be holomorphic in $\Disk(0,4)$ and zero-free in $\Disk(0,2)$.
Assume
$$
 \log\frac{\sup_{\Disk(0,4)}|h|}{\sup_{\Disk(0,1)}|h|}
 \le S.
$$
Then
$$
 |h(x)|\ge e^{-CS}\sup_{\Disk(0,1)}|h|,
\quad x\in I_0,
$$
and
$$
 \left|\frac{h'(x)}{h(x)}\right|
 \le CS,
\quad x\in I_0.
$$
\end{lemma}

\begin{proof}
Let
$$
 M_4 = \sup_{\Disk(0,4)}|h|,
\quad
 u(z) = \log\frac{M_4}{|h(z)|}.
$$
Since $h$ is zero-free in $\Disk(0,2)$, $u$ is nonnegative and harmonic
there. Choose $w\in\overline{\Disk(0,1)}$ such that
$$
 |h(w)| = \sup_{\overline{\Disk(0,1)}}|h|.
$$
Then $u(w)\le S$. By Harnack's inequality on compact subdisks of
$\Disk(0,2)$,
$$
 u(z)\le CS,
\quad |z|\le3/2.
$$
In particular, for $x\in I_0$,
$$
 |h(x)|\ge M_4e^{-CS}\ge e^{-CS}\sup_{\Disk(0,1)}|h|.
$$

For the logarithmic derivative, take a holomorphic branch of $\log h$ in
$\Disk(0,2)$, and put
$$
 F(z) = \log h(z)-\log h(0).
$$
The preceding estimate gives
$$
 \osc_{\Disk(0,3/2)}\log|h|\le CS,
$$
hence $|\operatorname{Re}F(z)|\le CS$ for $|z|\le3/2$. Thus a bound for
the real part of the holomorphic logarithm on the larger disk is available.
By the Borel-Carath\'eodory inequality, this gives
$\sup_{\Disk(0,1)}|F|\le CS$. Since
$\Disk(x,1/4)\subset\Disk(0,1)$ for $x\in I_0$, Cauchy's estimate gives
$|F'(x)|\le CS$. Since $F' = h'/h$, the proof is complete.
\end{proof}

\subsection{Proof of the effective zero-product proposition} \label{app:zp}

\begin{proof}[Proof of Proposition \ref{prop:zp}]
We begin with the global zero count. Let $\mathcal Z$ be the zeros of $f$
in $\Disk(0,2)$, counted with multiplicity, and let $N = |\mathcal Z|$.
Form the partial Blaschke product $B = B_{\mathcal Z}^{0,4}$. Lemma
\ref{lem:pb}, with $a = 0$, $R = 4$, and $s = 1$, gives
$$
 \sup_{\Disk(0,1)}|f|
 \le
 \left(\sup_{\Disk(0,1)}|B|\right)
 \sup_{\Disk(0,4)}|f|.
$$
For every zero $\zeta\in\mathcal Z$ we have $|\zeta|<2$, and Lemma
\ref{lem:bf} gives
$$
 \sup_{\Disk(0,1)}|B|\le (6/7)^N.
$$
Combining this with \eqref{G0} yields
$$
 1\le (6/7)^N e^{H_0},
$$
and hence
$$
 N\le C H_0.
$$

Next factor out the zeros in $\Disk(0,2)$:
$$
 h = \frac{f}{B_{\mathcal Z}^{0,4}}.
$$
The quotient is holomorphic in $\Disk(0,4)$ and zero-free in $\Disk(0,2)$.
Since $|B_{\mathcal Z}^{0,4}| = 1$ on $\partial\Disk(0,4)$ and
$|B_{\mathcal Z}^{0,4}|\le1$ in $\Disk(0,4)$,
$$
 \sup_{\Disk(0,4)}|h|\le\sup_{\Disk(0,4)}|f|,
\quad
 \sup_{\Disk(0,1)}|h|\ge\sup_{\Disk(0,1)}|f|\ge1.
$$
By \eqref{G0}, the hypothesis of Lemma
\ref{lem:zf} holds with $S = H_0$. Hence
$$
 |h(x)|\ge e^{-CH_0},
\quad x\in I_0.
$$
For $x\in I_0$ and $\zeta\in\Disk(0,2)$,
$$
 \left|4\frac{x-\zeta}{16-\overline\zeta x}\right|
 \ge c\min\{1,|x-\zeta|\}.
$$
Using $N\le CH_0$, we obtain
\begin{equation} \label{eq:allz}
 |f(x)|
 \ge
 e^{-CH_0}
 \prod_{\zeta\in\mathcal Z}\min\{1,|x-\zeta|\},
\quad x\in I_0.
\end{equation}

It remains to reduce the full product to the first $d$ nearest zeros. As
throughout, possible zeros on the auxiliary boundary are handled by replacing
the radius $1$ by $1-\eta$ or $1+\eta$, proving the estimate with constants
independent of $\eta$, and then letting $\eta\downarrow0$. We shall not
repeat this limiting step in the notation below.

Fix $x\in I_0$. The zeros of $f$ in $\Disk(x,1)$ are among the zeros in
$\Disk(0,2)$, because $|x|\le1/2$. Let
$$
 \zeta_1,\ldots,\zeta_M
$$
be these local zeros, counted with multiplicity and ordered by increasing
$|\zeta_j-x|$. We claim that
\begin{equation} \label{eq:tail}
 \prod_{j = d+1}^{M}|\zeta_j-x|
 \ge e^{-CH_1}
\end{equation}
whenever $M\ge d+1$. To prove the claim, first note that the number of
local zeros satisfies $M\le CH_1$. Indeed, applying Lemma
\ref{lem:pb} in $\Disk(x,4)$ to the zeros in
$\Disk(x,1)$, and using Lemma \ref{lem:bf} with
$R = 4$, $s = 1$, $\rho = 1$, gives
$$
 \sup_{\Disk(x,1)}|f|
 \le (8/15)^M\sup_{\Disk(x,4)}|f|.
$$
By \eqref{G1} with $r = 1$,
$$
 \sup_{\Disk(x,4)}|f|\le e^{H_1}\sup_{\Disk(x,1)}|f|,
$$
and therefore $M\le CH_1$.

Assume now $M\ge d+1$, and put
$$
 \varrho = |\zeta_{d+1}-x|.
$$
We first show that $\varrho>0$. If $\varrho=0$, then $f$ has a zero at
$x$ of multiplicity $k\ge d+1$. Thus
$$
 \sup_{\Disk(x,r)}|f| = O(r^k)
$$
as $r\downarrow0$, and hence
$$
 r^{-d}\sup_{\Disk(x,r)}|f| = O(r^{k-d})\to0.
$$
This contradicts \eqref{G1}, because $\sup_{\Disk(x,4)}|f|>0$. Therefore
$0<\varrho<1$.

We prove \eqref{eq:tail}. To avoid any boundary convention, fix
$\eta>0$ so small that
$$
 r_\eta = (1+\eta)\varrho < 1.
$$
The parameter $\eta$ is only used to avoid boundary issues. Set
$$
 M_4 = \sup_{\Disk(x,4)}|f|,
\quad
 M_{r_\eta} = \sup_{\Disk(x,r_\eta)}|f|.
$$
By \eqref{G1} with $r = r_\eta$,
$$
 \frac{M_{r_\eta}}{M_4}
 \ge
 e^{-H_1}r_\eta^d.
$$
On the other hand, form the partial Blaschke product in $\Disk(x,4)$ over
all local zeros. Lemma \ref{lem:pb}, applied with $s = r_\eta$, gives
$$
 \frac{M_{r_\eta}}{M_4}
 \le
 \sup_{\Disk(x,r_\eta)}|B|.
$$
For $z\in\Disk(x,r_\eta)$, write $u=z-x$ and $a_j=\zeta_j-x$. Since
$|u|<r_\eta<1$ and $|a_j|<1$, the denominator of each Blaschke factor in
$\Disk(x,4)$ is at least $15$ in modulus. For $1\le j\le d$,
$|a_j|\le\varrho<r_\eta$, hence
$$
 |u-a_j|\le |u|+|a_j|\le 2r_\eta.
$$
For $j\ge d+1$, $|a_j|\ge\varrho$, and therefore, for $0<\eta\le1$,
$$
 |u-a_j|\le r_\eta+|a_j|
 \le (2+\eta)|a_j|
 \le 3|a_j|.
$$
Consequently
$$
 \sup_{\Disk(x,r_\eta)}|B|
 \le
 C^M r_\eta^d\prod_{j=d+1}^{M}|\zeta_j-x|.
$$
Combining the last two displays gives
$$
 \prod_{j=d+1}^{M}|\zeta_j-x|
 \ge
 e^{-H_1}C^{-M}.
$$
Since $M\le CH_1$, this proves \eqref{eq:tail}.

We now pass from the full product to the first $d$ nearest zeros. All zeros of
$\mathcal Z$ outside $\Disk(x,1)$ contribute the factor $1$ to
$\min\{1,|x-\zeta|\}$. Zeros with $|\zeta-x| = 1$ are not included among
the local zeros and also contribute the factor $1$. Hence
$$
 \prod_{\zeta\in\mathcal Z}\min\{1,|x-\zeta|\}
 =
 \prod_{j=1}^{M}|\zeta_j-x|.
$$
If $M\le d$, then all local zeros appear among the first $\min\{d,N\}$ zeros
of $\mathcal Z$ ordered by distance from $x$, and every remaining zero
contributes the factor $1$. If $M\ge d+1$, the tail estimate \eqref{eq:tail}
gives
$$
 \prod_{\zeta\in\mathcal Z}\min\{1,|x-\zeta|\}
 =
 \prod_{j=1}^{M}|\zeta_j-x|
 \ge
 e^{-CH_1}\prod_{j=1}^{d}|\zeta_j-x|.
$$
In this case $\zeta_1,\ldots,\zeta_d$ are exactly the first $d$ zeros in the
global ordering, because local zeros have distance $<1$ and nonlocal zeros have
distance at least $1$. Combining both cases with \eqref{eq:allz} gives the
stated estimate.
\end{proof}

\subsection{Cartan covering}

\begin{lemma}[Cartan covering] \label{lem:cartan}
Let $N\ge0$, and let $w_1,\ldots,w_N\in\C$ be labelled points, with
repetitions allowed. Let $0<h\le1$. If $N\ge1$, then for each $z\in\C$
let
$$
 d_1(z)\le\cdots\le d_N(z)
$$
be the numbers $|z-w_1|,\ldots,|z-w_N|$ arranged in nondecreasing order.
There are absolute constants $c,C>0$ and a family of at most $N$ disks
$\{D_\nu\}$ such that
$$
 \sum_\nu\rad(D_\nu)\le Ch
$$
and, for every $z\notin\bigcup_\nu D_\nu$,
$$
 d_j(z)\ge c\frac{hj}{N},
\quad 1\le j\le N.
$$
Consequently, for every integer $d\ge1$, if $q = \min\{d,N\}$, then outside
$\bigcup_\nu D_\nu$,
$$
 \prod_{j = 1}^{q}\min\{1,d_j(z)\}
 \ge e^{-CN}h^d.
$$
When $N = 0$, the covering is empty and the product is $1$.
\end{lemma}

\begin{proof}
The case $N = 0$ is immediate. Assume $N\ge1$. Put
$$
 r_j = \frac{hj}{4N},
\quad 1\le j\le N.
$$
For each $j$ and each subset $S\subset\{1,\ldots,N\}$ with $|S| = j$ such
that
$$
 \bigcap_{i\in S}\overline{\Disk(w_i,r_j)}\ne\varnothing,
$$
choose a point $c_S$ in the intersection and form
$\Delta_S = \overline{\Disk(c_S,r_j)}$. Thus $\Delta_S$ contains the labelled
points $w_i$, $i\in S$.

Run the following finite greedy selection. Starting at level $j = N$ and then
proceeding through $j = N-1,\ldots,1$, list all disks $\Delta_S$ at that
level in an arbitrary order. Select $\Delta_S$ if it is disjoint from all
previously selected disks. Let the selected disks be
$\Delta_\nu = \overline{\Disk(c_\nu,r_{j_\nu})}$. They are pairwise disjoint,
and their distinguished labelled sets are disjoint; hence
$$
 \sum_\nu j_\nu\le N.
$$
Therefore
$$
 \sum_\nu \rad(\Delta_\nu)
 = \frac h{4N}\sum_\nu j_\nu
 \le \frac h4.
$$
Let
$$
 D_\nu = \Disk(c_\nu,6r_{j_\nu}).
$$
Then $\sum_\nu\rad(D_\nu)\le Ch$ and the number of disks is at most $N$.

We prove the ordered-distance estimate. Suppose that a closed disk
$\overline{\Disk(z,r_j)}$ contains at least $j$ labelled points. Choose a
set $S$ of $j$ corresponding labels. The associated disk $\Delta_S$
intersects $\overline{\Disk(z,r_j)}$. When $\Delta_S$ was considered in
the greedy selection, either it was selected, or it met a previously selected
disk $\Delta_\nu$ with $j_\nu\ge j$. In both cases there is a selected
$\Delta_\nu$ with $j_\nu\ge j$ such that
$\Delta_S\cap\Delta_\nu\ne\varnothing$. Since
$r_{j_\nu}\ge r_j$, this implies
$$
 \overline{\Disk(z,r_j)}\subset \overline{\Disk(c_\nu,5r_{j_\nu})}
 \subset D_\nu.
$$
Thus, if $z\notin\bigcup_\nu D_\nu$, no disk $\overline{\Disk(z,r_j)}$
contains $j$ labelled points. Therefore
$$
 d_j(z)\ge r_j = \frac{hj}{4N},
\quad 1\le j\le N.
$$

Let $q = \min\{d,N\}$. Outside the exceptional disks,
$$
 \prod_{j = 1}^{q}\min\{1,d_j(z)\}
 \ge
 \prod_{j = 1}^{q}\frac{hj}{4N}
 =
 \left(\frac h{4N}\right)^q q!.
$$
If $d\le N$, Stirling's lower bound gives
$$
 \left(\frac h{4N}\right)^d d!
 \ge
 h^d\exp\left(-d\log\frac{4eN}{d}\right)
 \ge e^{-CN}h^d.
$$
If $N<d$, then
$$
 \left(\frac h{4N}\right)^N N!
 \ge h^N(4e)^{-N}
 \ge e^{-CN}h^d,
$$
because $0<h\le1$ and $N<d$. This proves the lemma.
\end{proof}

%%%%%%%%%%%%%%%%%%%%%%%%%%%%%%%%%%%%%%%%%%%
\section{The weak logarithmic derivative estimate} \label{app:wld}

This appendix proves Lemma \ref{lem:wld}. The proof has four ingredients:
disk zero counting, Blaschke factorization, the zero-free quotient estimate
from Lemma \ref{lem:zf}, and a weak Cauchy estimate. The last ingredient is a
standard consequence of the weak $(1,1)$ inequality for the Hilbert transform.
We emphasize that the zero-counting lemma below uses the disk-growth
consequence of the classical interval Tur\'an inequality, Lemma
\ref{lem:edg}. Thus the weak logarithmic-derivative estimate also ultimately
depends on the same classical interval Tur\'an input, and on no measurable
Tur\'an-Nazarov estimate.

\begin{lemma}[Disk zero counting for exponential polynomials] \label{lem:zc}
Let
$$
 q(z) = \sum_{k = 1}^m a_k e^{\mu_kz}
$$
be a nonzero exponential polynomial of order at most $m$, and put
$\sigma = \max_k|\mu_k|$. Then, for every $z_0\in\C$ and every $r>0$,
the number of zeros of $q$ in $\overline{\Disk(z_0,r)}$, counted with
multiplicity and understood through the usual limiting convention, is at most
$$
 C(m+\sigma r).
$$
\end{lemma}

\begin{proof}
It is enough to count zeros in $\Disk(z_0,r)$, since the closed-disk version
follows by replacing $r$ by $r+\eta$ and letting $\eta\downarrow0$. Let
$Z$ be the multiset of zeros in $\Disk(z_0,r)$, and put $N = |Z|$. Set
$R = 8r$. By Lemma \ref{lem:edg},
$$
 \log\frac{\sup_{\Disk(z_0,R)}|q|}{\sup_{\Disk(z_0,r)}|q|}
 \le C(m+\sigma r).
$$
Form the partial Blaschke product in $\Disk(z_0,R)$ over the zeros in $Z$.
By Lemmas \ref{lem:pb} and
\ref{lem:bf},
$$
 \sup_{\Disk(z_0,r)}|q|
 \le \theta^N\sup_{\Disk(z_0,R)}|q|
$$
with an absolute $\theta<1$. Taking logarithms gives the estimate.
\end{proof}

\begin{lemma}[Weak Cauchy estimate] \label{lem:cauchy}
For $\zeta_1,\ldots,\zeta_N\in\C$ and $y>0$,
$$
 \left|
 \left\{x\in\R:
 \left|\sum_{j = 1}^N\frac1{x-\zeta_j}\right|>y
 \right\}
 \right|
 \le
 \frac{CN}{y},
$$
where real poles are omitted from the level set.
\end{lemma}

\begin{proof}
This is a standard consequence of the weak $(1,1)$ inequality for the
Hilbert transform. We spell out the reduction because the estimate is used in
an essential way below. For a real pole, the kernel $(x-a)^{-1}$ is, up to the
normalization of the Hilbert transform, the Hilbert transform of the point mass
at $a$. Thus all real poles contribute the Hilbert transform of a finite signed
measure of total variation at most $CN$.

For a non-real pole $\zeta = a+ib$,
$$
 \frac1{x-\zeta}
 =
 \frac{x-a}{(x-a)^2+b^2}
 +i\frac{b}{(x-a)^2+b^2}.
$$
The imaginary part has $L^1(\R)$-norm $\pi$. The real part is, up to sign
and normalization, the Hilbert transform of the Poisson kernel
$$
 \frac{|b|}{(x-a)^2+b^2},
$$
whose $L^1(\R)$-norm is $\pi$. Hence, after summing all real and non-real
poles, we can write
$$
 \sum_{j=1}^N\frac1{x-\zeta_j}
 = H\mu(x)+P(x)
$$
outside the real poles, where $\mu$ is a finite signed measure with
$\|\mu\|\le CN$ and $P\in L^1(\R)$ satisfies
$\|P\|_1\le CN$. The weak $(1,1)$ inequality for the Hilbert transform and
Chebyshev's inequality give
$$
 |
 \{x:|H\mu(x)|>y/2\}
 |
 \le \frac{C\|\mu\|}{y},
\quad
 |
 \{x:|P(x)|>y/2\}
 |
 \le \frac{2\|P\|_1}{y}.
$$
Since $|H\mu+P|>y$ implies one of the two inequalities on the left, the
claimed estimate follows; see, for example, \cite[Chapter~3]{Duoandikoetxea}.
\end{proof}

\begin{proof}[Proof of Lemma \ref{lem:wld}]
If $u\le C_0(m+\sigma)$, with $C_0$ sufficiently large, then the desired
estimate is trivial because $|I_0| = 1$. Hence assume
$$
 u>C_0(m+\sigma).
$$

Let $Z = \{\zeta_1,\ldots,\zeta_N\}$ be the multiset of zeros of $q$ in
$\Disk(0,2)$, counted with multiplicity, and set
$$
 B = B_Z^{0,4},
\quad
 h = q/B.
$$
By Lemma \ref{lem:zc},
$$
 N\le C(m+\sigma).
$$
The quotient $h$ is holomorphic in $\Disk(0,4)$ and zero-free in
$\Disk(0,2)$. Since $|B| = 1$ on $\partial\Disk(0,4)$ and $|B|\le1$
in $\Disk(0,4)$,
$$
 \sup_{\Disk(0,4)}|h|\le\sup_{\Disk(0,4)}|q|,
\quad
 \sup_{\Disk(0,1)}|h|\ge\sup_{\Disk(0,1)}|q|.
$$
Lemma \ref{lem:edg}, with $R = 4$ and $r = 1$, gives
$$
 \log\frac{\sup_{\Disk(0,4)}|h|}{\sup_{\Disk(0,1)}|h|}
 \le C(m+\sigma).
$$
By Lemma \ref{lem:zf},
$$
 \left|\frac{h'(t)}{h(t)}\right|
 \le C(m+\sigma),
\quad t\in I_0.
$$

Away from the real zeros of $q$,
$$
 \frac{q'}q = \frac{B'}B+\frac{h'}h.
$$
For one Blaschke factor
$$
 b_\zeta(z) = 4\frac{z-\zeta}{16-\overline\zeta z},
$$
we have
$$
 \frac{b_\zeta'}{b_\zeta}(z)
 =
 \frac1{z-\zeta}
 +
 \frac{\overline\zeta}{16-\overline\zeta z}.
$$
For $t\in I_0$ and $|\zeta|<2$, the second term is bounded absolutely by
an absolute constant. Therefore
$$
 \left|\frac{q'(t)}{q(t)}\right|
 \le
 \left|\sum_{j = 1}^N\frac1{t-\zeta_j}\right|
 +C(m+\sigma).
$$
Since $u>C_0(m+\sigma)$, with $C_0$ large enough, the level set
$\{|q'/q|>u\}$ is contained in
$$
 \left\{t\in I_0:
 \left|\sum_{j = 1}^N\frac1{t-\zeta_j}\right|>u/2
 \right\}.
$$
By Lemma \ref{lem:cauchy}, its measure is at most
$$
 \frac{CN}{u}
 \le
 \frac{C(m+\sigma)}{u}.
$$
This proves the lemma.
\end{proof}

%%%%%%%%%%%%%%%%%%%%%%%%%%%%%%%%%%%%%%%%%%%
\section*{Acknowledgements}

The author is grateful to Philippe Jaming and Evgueni Abakoumov for encouraging
comments on an earlier version of this work.

%%%%%%%%%%%%%%%%%%%%%%%%%%%%%%%%%%%%%%%%%%%

\end{document}